\documentclass[12pt]{article}
\usepackage[top=2.54cm, bottom=2.54cm, left=3.18cm, right=3.18cm]{geometry}
\usepackage{amsmath}
\usepackage{amssymb}
\usepackage{amsthm}
\usepackage{fancyhdr}
\usepackage{latexsym}
\usepackage{mathrsfs}
\usepackage{wasysym}
\usepackage{float}
\usepackage{graphicx}
\usepackage[compress]{cite}
\usepackage{array}
\usepackage{tikz}
\usepackage{caption,enumitem}
\usepackage{xcolor}
\usepackage{subcaption}
\allowdisplaybreaks[4]

\usepackage{mathrsfs}
\usetikzlibrary{arrows}

 \usepackage[backref=page]{hyperref}

\usepackage{algorithm}
\usepackage{algpseudocode}

\usepackage{booktabs,tabularx,ragged2e}

\newtheorem{theorem}{\bf Theorem}[section]

\newtheorem{lemma}[theorem]{Lemma}

\newtheorem{problem}[theorem]{Problem}

\newtheorem{proposition}{Proposition}

\begin{document}

\title{Max-Bisections of graphs without perfect matching}

\author{Jianfeng Hou$^{1}$\thanks{Research supported by National Key R\&D Program of China (Grant No. 2023YFA1010202), National Natural Science Foundation of China (Grant No. 12071077), the Central Guidance on Local Science and Technology Development Fund of Fujian Province (Grant No. 2023L3003). Email: \texttt{jfhou@fzu.edu.cn}}, Shufei Wu$^2$\thanks{Email: \texttt{shufeiwu@hotmail.com}; Research supported by the Natural Science Foundation of China (No. 11801149), Key
Research Funds for the Universities of Henan Province (Grant No. 25A110003)}, Yuanyuan Zhong$^2$\thanks{Email: \texttt{zhongyy02023@126.com}},\\
\small{$^1$Center of Discrete Mathematics, Fuzhou University, Fujian, China, 350116}\\
{\small $^2$School of Mathematics and Information Science, Henan Polytechnic University,
Henan, China, 454003}}

\maketitle

\begin{abstract}
A bisection of a graph is   a bipartition of its vertex set such that the two resulting parts differ in size by
at most 1, and its size is the number of edges that connect vertices in the two parts. The perfect matching condition and  forbidden even cycles  subgraphs are essential in finding  large bisections of graphs. In this paper, we show that the perfect matching condition can be replaced by the minimum degree condition. Let $C_{\ell}$ be a cycle of length $\ell$ for $\ell\ge 3$, and let $G$ be  a $\{C_4, C_6\}$-free graph with $m$ edges and minimum degree at least 2.   We prove that  $G$ has a bisection of size at least $m/2+\Omega\left(\sum_{v\in V(G)}\sqrt{d(v)}\right)$.  As a corollary,   if  $G$ is also  $C_{2k}$-free for  $k\ge3$, then $G$ has a bisection of size at least
$m / 2+\Omega\left(m^{(2 k+1) /(2 k+2)}\right)$, thereby confirming a conjecture proposed by
 Lin and Zeng [J.  Comb. Theory A, 180 (2021), 105404].
\end{abstract}



\section{Introduction}
Let $G$ be a graph and $\mathcal{F}$ be a family of graphs.  We say $G$ is \emph{$\mathcal{F}$-free} if it does not contain any member of $\mathcal{F}$ as a subgraph. If $\mathcal{F}=\{H\}$, then we use $H$-free instead of $\{H\}$-free.  A \emph{bipartition} $(V_0,V_1)$ of $G$ is a partition of $V(G)$ satisfies  $V_0\cap V_1=\emptyset$ and $V_0\cup V_1=V(G)$. The \emph{size} of $(V_0,V_1)$, denoted by $e_G(V_0, V_1)$, is the number of  edges with one end in $V_0$ and the other in $V_1$. A \emph{bisection} of $G$ is a bipartition $(V_0, V_1)$ of $G$ with  $||V_0|-|V_1||\le 1$. The well-known \emph{Max-Cut problem}  is to find a bipartite $(V_0,V_1)$ of $G$  that maximizes $e_G(V_0, V_1)$.  We will drop the subscript when the confusion is unlikely.

Considering a random bipartition of a graph $G$ with $m$ edges, it is easy to see $G$ has a bipartition of size at least $m/2$. Answering a question of Erd\H{o}s, the bound was improved to $m/2+(\sqrt{8m+1}-1)/8$ by Edwards \cite{Edw1973, Edw1975}. Note that the bound is tight for complete graphs with odd orders. A natural step is to study  Max-Cuts of $H$-free graphs for a fixed graph $H$. It was initialed   by Erd\H{o}s and Lov\'{a}sz (see \cite{Erd1979}) who showed  every triangle-free graph with $m$ edges has a bipartition of size  at least $m/2+\Omega\left(m^{2 / 3}(\log m / \log \log m)^{1 / 3}\right)$. A breakthrough on this topic was given by Shearer \cite{She1992} who showed that every triangle-free $G$ with $n$ vertices, $m$ edges and degree sequence $d_{1}, d_{2},\ldots ,d_{n}$ admits  a bipartition of size at least
\begin{equation}\label{sh-bound}
\frac m2+\Omega\left(\sum_{i=1}^{n}\sqrt{d_{i}}\right).
\end{equation}
As a corollary, such a graph has a bipartition of size at least $m/2+\Omega(m^{3/4})$. The tight  bound on Max-Cut of triangle-free graphs, given by Alon \cite{Alo1996}, is $m/2+\Theta(m^{4/5})$.

Usually, the lower bound of \eqref{sh-bound}, known as \emph{Shearer's bound}, serves as  a pivotal foundation that has facilitated the derivation of several intriguing results in the field. For example, Alon, Krivelevich and Sudakov \cite{Alo2005} showed \eqref{sh-bound} holds for graphs with sparse neighborhood and proved every $C_{2k}$-free graph with $m$ edges has a bipartition of size  at least
$m/2+\Omega(m^{(2k+1)/(2k+2)})$, where $C_{\ell}$ denotes a cycle of length $\ell\ge 3$. An  exciting result, given by Glock, Janzer and Sudakov \cite{Glock2023}, is such a  bound  also holds for $C_{2k+1}$-free graphs. We refer the reader to \cite{Boll2002,Erd1967,Ma2019,Sco2005,Wu2023,Zen2017,Zen2018} for further problems and results in this direction.

In this paper, we focus on the \emph{Max-Bisection problem}: Find a bisection of a given graph  that maximizes its size. Compared to bipartition, bisections of graphs are more complicated to analyze. For example, as noticed in \cite{Erd1997,Pol1994}, the Edwards' bound implicitly implies that a connected graph $G$ with $n$ vertices and $m$ edges admits a bipartition of size at least $m/2+(n-1)/4$. However,  each bisection of the complete bipartite graph $K_{d,n-d}$ with $m$ edges has size at most $\lceil (m+d^2)/2 \rceil$. Motivated by Max-Cuts, Bollob\'as and Scott \cite{Bol2002} asked the following:
\begin{problem}[Bollob\'as and Scott \cite{Bol2002}]\label{B-S-Problem}
What are the largest and smallest cuts that we can guarantee with bisections of graphs?
\end{problem}

The  majority of  results are to find a bisection of size at least $m/2+cn$ for some $c>0$ in  $H$-free graphs with $n$ vertices and $m$ edges \cite{Fan2018,   jin2019, Lee2013,Hou2017,Ji2019,Liu2024,Xu2010,Xu2010balanced,Xu2008,Xu2014}. Forbidden even cycle subgraphs seems reasonable
 for Problem \ref{B-S-Problem}. A typical example, given by  Hou and Yan \cite{Hou2020}, shows that every connected $C_4$-free graph $G$ with $n$ vertices, $m$ edges and   minimum degree at least 2 admits a bisection of size at least $m/2+(n-1)/4.$ It is natural to give sufficient conditions  satisfying the Shearer's bound. The first step was given by Lin and Zeng \cite{Lin2021}  using a bisection version of Shearer's randomized algorithm.

\begin{theorem}[Lin and Zeng\cite{Lin2021}]\label{{C_{4},C_{6}}-free}
Let $G$ be a  $\{C_{4},C_{6}\}$-free graph  with $n$ vertices, $m$ edges and degree sequence $d_{1}, d_{2},\ldots ,d_{n}$. If $G$ has a perfect matching, then  $G$ admits a bisection of size at least
$m/2+\Omega(\sum_{i=1}^{n}\sqrt{d_{i}}).$
\end{theorem}

Combining Theorem \ref{{C_{4},C_{6}}-free} and a standard degenerate argument, they also proved
\begin{theorem}[Lin and Zeng \cite{Lin2021}]\label{c2k}
For any integer $k \ge3$, let $G$ be a $\{C_4,C_6,C_{2k}\}$-free graph with $m$ edges. If $G$ has a perfect
matching, then  $G$ admits a bisection of size at
least
$m/2+\Omega(m^{(2k+1)/(2k+2)})$.
\end{theorem}

We remark that the proof of Theorem \ref{{C_{4},C_{6}}-free} relies  heavily on the existence  of  perfect matchings in graphs. In \cite{Lin2021}, the authors conjectured that the existence  of  perfect matchings can be replaced by the minimum degree condition, and posed the following problem.

\begin{problem}[Lin and Zeng \cite{Lin2021}]\label{pro1.1}
For a fixed integer $k \geq 3$, does every $\{C_4,C_6,C_{2k}\}$-free graph with $m$ edges and minimum degree at least 2 admit a bisection of size at least
$m/2+\Omega(m^{(2k+1)/(2k+2)})$?
\end{problem}

In this paper, we extend Theorem \ref{{C_{4},C_{6}}-free} to graphs without a perfect matching and provide a comprehensive confirmation of  Problem \ref{pro1.1}.

\begin{theorem}\label{the1.1}
Let $G$ be a connected $\{C_{4},C_{6}\}$-free graph with $n$ vertices, $m$ edges and degree sequence $d_{1}\ge  d_{2}\ge \cdots \ge d_{n}\ge 2$. Then there is a constant $\xi>0$ such that $G$ admits a bisection of size at least
\[\frac m2+\xi\sum_{i=1}^{n}\sqrt{d_{i}}.\]
\end{theorem}

\begin{theorem}\label{the1.2}
For any fixed integer $k \geq 3$, let $G$ be a connected $\{C_4,C_6,C_{2k}\}$-free graph with $m$ edges and minimum degree at least 2. Then there is a constant $c(k)>0$ such that $G$ admits  a bisection of size at least
\[\frac m2+c(k) m^{(2 k+1) /(2 k+2)}.\]
\end{theorem}

Considering $(n,d,\lambda)$-graphs and using the same argument as in \cite{Lin2021}, we remark  that both bounds are tight up to the value of $\xi$ in Theorem \ref{the1.1}, and  $c(k)$ for $2k\in\{6,10\}$ in Theorem \ref{the1.2}. Moreover, the minimum degree condition in both theorems is necessary by considering a star.  We will prove Theorem \ref{the1.1} in Section \ref{SEC:pf-thm-1} and Theorem \ref{the1.2} in Section \ref{SEC:pf-thm-2}.


\section{Preliminaries}\label{SEC:prelim}

\subsection{A random bisection algorithm}

In this subsection, we introduce a simply random algorithm to get a bisection of large size. It was initialed by Shearer  \cite{She1992} for Max-Cuts of triangle-free graphs, and extended by Lin and Zeng \cite{Lin2021} for Max-Bisections of $\{C_4, C_6\}$-free graphs with a perfect matching. In the following of the paper, we will omit unnecessary parentheses and commas in the set representation and for example, we write a pair $uv$ instead of $\{u,v\}$.

Let $G$ be a graph with $n$ vertices, where $n$ is even. We use $e(G)$ to denote the number of edges in $G$. For two disjoint subsets $S, T$ of $V(G)$, we use $E_G(S, T)$ denote the set of edges with one end in $S$ and the other in $T$, and $e_G(S,T)=|E_G(S, T)|$. Let $M$ be a maximum matching of $G$. Then $W = V(G)\setminus V(M_1)$ is an independent set. Assume $|W|=2k$. Partition $W$ into $k$  vertex-disjoint pairs $w_1w'_1,\ldots, w_{k}w'_{k}$, and call $\mathbb{M}=M\cup \{w_1w'_1,\ldots, w_{k}w'_{k}\}$ a \emph{quasi-perfect matching} of $G$. For each vertex $v \in V(G)$, we always use $v'$ to denote the vertex paired with $v$ under $\mathbb{M}$. That is, $vv'\in \mathbb{M}$. A folklore method to get a bisection of $G$ is to choice a quasi-perfect matching $\mathbb{M}$ of $G$, and then let either  $h(v)=0$  and  $h\left(v^{\prime}\right)=1$, or  $h(v)=1$  and  $h\left(v^{\prime}\right)=0$ for each pair $vv^{\prime}\in \mathbb{M}$.  Define  $V_{i}=h^{-1}(i)$ for $i=0,1$. Clearly, $(V_{0},V_{1})$ is a bisection of  $G$. For $i\in\{0,1\}$ and a pair $vv^{\prime}\in \mathbb{M}$,  let
\[
N_{i}(v)=\{u \in N(v):uv \notin \mathbb{M}, h(u)=i\},
\]
 and
\begin{align}\label{Def:sigma(vv')-stable}
\sigma(vv^{\prime})=|N_{1-h(v)}(v)|+|N_{1-h(v')}(v')|-|N_{h(v)}(v)|-|N_{h(v')}(v')|.
\end{align}
Call $vv^{\prime}$ \emph{stable} under $(V_{0},V_{1})$ if $\sigma(vv^{\prime}) \geq 0$, otherwise call it \emph{active}. Clearly, switching the value of $h$ on  active pairs can increase the size of the bisection.

Algorithm \ref{ALGO:Two-stage-Random-Algorithm} gives a random bisection of $G$ in time $O(|V(G)|+|E(G)|)$.

\begin{algorithm}[H]
\renewcommand{\thealgorithm}{1}
\caption{\textsc{Two-stage Random Algorithm}}\label{ALGO:Two-stage-Random-Algorithm}
\begin{algorithmic}

\State \textbf{Input:}
A graph $G$ with $n$ vertices and a quasi-perfect matching $\mathbb{M}$ of $G$.

\State \textbf{Output:}
A random bisection $(V_{0},V_{1})$ of $G$.
\For{each $vv^{\prime}\in \mathbb{M}$}
\State  Set randomly and independently either  $h(v)=0$  and  $h\left(v^{\prime}\right)=1$, or  $h(v)=1$  and  $h\left(v^{\prime}\right)=0$,  where both choices are equally likely.
\EndFor
\State  Define  $U_{i}=h^{-1}(i)$ for $i=0,1$
\For{each $vv^{\prime}\in \mathbb{M}$}
\State Calculate $\sigma(vv^{\prime})$.
\If{$vv'$ is active under $(U_0,U_1)$}
\State Set randomly and independently again, either  $h^{\prime}(v)=0$  and  $h^{\prime}\left(v^{\prime}\right)=1$, or  $h^{\prime}(v)=1$  and  $h^{\prime}\left(v^{\prime}\right)=0$, where both choices are equally likely.
\Else
\State Set  $h^{\prime}(v)=h(v)$  and  $h^{\prime}\left(v^{\prime}\right)=h\left(v^{\prime}\right)$.
\EndIf
\EndFor
\State Define  $V_{i}=\left(h^{\prime}\right)^{-1}(i)$  for $i=0,1$.
\end{algorithmic}
\end{algorithm}

The following  is key for Theorem \ref{the1.1}, whose  proof  will be postponed in Section \ref{SEC:pf-main-lemma}.
\begin{lemma}\label{th-main}
Let $G$ be a $\{C_{4},C_6\}$-free graph with even number of vertices and minimum degree at least 2. Then there exist a constant $\epsilon>0$  and a quasi-perfect matching $\mathbb{M}$ of $G$ such that the following holds: Suppose that $(V_{0},V_{1})$ is a random bisection given by Algorimth \ref{ALGO:Two-stage-Random-Algorithm}. For each edge $uv \in E(G)\setminus E(\mathbb{M})$, if $uv$ is the unique edge connecting $uu'$ and $vv'$, then  the  probability that $uv$ lies in $(V_0,V_1)$ is at least
\[\frac{1}{2}+\epsilon\left(\frac{1}{\sqrt{d(u)+d(u')}}+\frac{1}{\sqrt{d(v)+d(v')}}\right).\]
 \end{lemma}

 \subsection{Useful lemmas}
In this subsection, we list some useful lemmas. We begin with a lower bound on Max-Bisections of $C_4$-free graphs, which will be used in the proof of Theorem \ref{the1.1} for sparse case. 
 \begin{theorem}[Hou and Yan \cite{Hou2020}]\label{cn}
Every connected $C_4$-free graph $G$ with $n$ vertices, $m$ edges and  minimum degree at least 2 admits a bisection of size at least $m/2+(n-1)/4$. 
\end{theorem}

Note that Algorithm \ref{ALGO:Two-stage-Random-Algorithm} is a random algorithm, and the stability of  pairs in some quasi-perfect matching can be expressed into the sum of some binomial coefficients.   For integers $N\ge r\ge 1$, let 
\[
B(N,r)=\frac{1}{2^{N}}\sum_{i=0}^{\lfloor r \rfloor}\binom{N}{i} \text{ and } \Phi(t_1,t_2)=B\left(s_{1}, \frac{s_{1}-t_{1}}{2}\right)B\left(s_{2}, \frac{s_{2}-t_{2}}{2}\right).
\]
Note that $B(N,r)$ represents the probability that at most $\lfloor r \rfloor$ of $N$ random coin flips are heads.  We need the following lemmas in the proof of Lemma \ref{th-main}.

\begin{lemma}[Lin and Zeng \cite{Lin2021}]\label{lem3.5}
Let $a, b, c, d$ be  positive integers such that $2c \geq a-1$ and  $2d \geq b-1$. Then
\[B(a,c)B(b,d)-B(a,c-1)B(b,d-1)=\Theta\left(\frac{1}{2^{a}}\binom{a}{c}+\frac{1}{2^{b}}\binom{b}{d}\right)\]
and
\[\frac{1}{2^{a}}\binom{a}{\lfloor\frac{a}{2}\rfloor}+\frac{1}{2^{b}}\binom{b}{\lfloor\frac{b}{2}\rfloor}=
\frac{1}{2^{a}}\binom{a}{\lfloor\frac{a+1}{2}\rfloor}+\frac{1}{2^{b}}\binom{b}{\lfloor\frac{b+1}{2}\rfloor}=
\Theta\left(\frac{1}{\sqrt{a+1}}+\frac{1}{\sqrt{b+1}}\right).\]
\end{lemma}

\begin{lemma}[Wu and Xiong \cite{Wu2024}]\label{lembb-bb}
Let 
\[
\Lambda:= \Phi(t,t)+\Phi(-t-2,-t-2)-\Phi(t+2,-t)-\Phi(-t,t+2). 
\]
$\mathrm{(i)}$ If  $t\ge 0$, then $\Lambda\geq \Phi(-t,-t)-\Phi(-t+2,-t+2).$

\noindent$\mathrm{(ii)}$ If  $t< 0$, then $\Lambda\geq  \Phi(t,t)-\Phi(t+2,t+2).$
\end{lemma}

\begin{lemma}[Wu and Zhong \cite{Wu2025}]\label{lembb-bb2}
Let 
\[
\Lambda:= \Phi(t,t-2)+\Phi(-t,-t-2)-\Phi(t+2,-t+2)-\Phi(-t+2,t+2). 
\]
$\mathrm{(i)}$ If $t\ge 0$, then $\Lambda\geq \Phi(-t,-t)-\Phi(-t+2,-t+2).$ 

\noindent$\mathrm{(ii)}$ If $t< 0$, then $\Lambda\geq \Phi(t,t)-\Phi(t+2,t+2).$

\noindent$\mathrm{(iii)}$ For any integer $t$,
\begin{align*}
\Phi(t-2,t)+\Phi(-t-2,-t)-\Phi(t,-t)-\Phi(-t,t) \geq  0.
\end{align*}

\end{lemma}
\begin{lemma}\label{lembb-bb3}
For any integer $t$,
\begin{align*}
&\Phi(t-2,t+2)+\Phi(t+2,t-2)+\Phi(-t-4,-t)+\Phi(-t,-t-4)\\
-&\Phi(t,-t-2)-\Phi(-t+2,t+4)-\Phi(-t-2,t)-\Phi(t+4,-t+2) \geq 0.
\end{align*}
\end{lemma}
 For the sake of coherence, we enclose the proof of Lemma \ref{lembb-bb3} in  Appendix.



\section{Proof of Theorem \ref{the1.1}}\label{SEC:pf-thm-1}
In this section, we prove Theorem \ref{the1.1}. Let $G$ be a connected $\{C_{4},C_{6}\}$-free graph with $n$ vertices, $m$ edges,  minimum degree at least 2 and degree sequence $d_{1}, d_{2},\ldots ,d_{n}$.  We may assume that $n$ is sufficiently large. Let $\xi=\min\{\frac{1}{32}, \frac{\epsilon}{2}\}$, where $\epsilon$ is given by Lemma \ref{th-main}. If $n\ge\frac{1}{8}  \sum_{i=1}^n \sqrt{d_i}+1$, then by Theorem \ref{cn}, $G$ has a bisection of size at least
\[ \frac{m}{2}+\frac{n-1}{4} \ge \frac{m}{2}+\xi \sum_{i=1}^n \sqrt{d_i},\]
and we are done.

Suppose $\sum_{i=1}^n \sqrt{d_i}>8(n-1)$. We may assume that $n$ is even, since otherwise, we can add a new vertex $x$ in $G$ and connect $x$ to two vertices in $G$ avoiding $C_4$ and $C_6$. Let $\mathbb{M}=\{e_1,\ldots,e_{n/2}\}$ be the quasi-perfect matching given by Lemma \ref{th-main} and $(V_{0},V_{1})$  be a random bisection  by Algorithm \ref{ALGO:Two-stage-Random-Algorithm}. To bound the size of $(V_{0},V_{1})$, it suffices to calculate the probability that each edge ends up $(V_{0},V_{1})$.

We claim that for each pair  $uu',vv'\in \mathbb{M}$, $e(uu',vv')\le 2$. 
Let $M$ be the maximum matching contained in $\mathbb{M}$. If both $uu'$ and $vv'$ are non-edges of $G$, then $e(uu',vv')=0$. If $uu',vv'\in E(G)$ and assume $uv\in E(G)$, then $u'v'\notin E(G)$ and at most one of $\{u'v,uv'\}$ belongs to $E(G)$ by the freeness of $C_4$. Otherwise, we may assume that $uu', uv\in E(G)$ and  $vv'\notin E(G)$. In this case, we also have $u'v'\notin E(G)$ and at most one of $\{u'v,uv'\}$ belongs to $E(G)$ by the maximum  of $M$.

For $1\le i<j\le n/2$, let $E_{ij}$ denote the set of edges between $e_i$ and $e_j$, and let $E_t$ be the union of  $E_{ij}$ with $|E_{ij}|=t$ for $t=1,2$. Then $E(G)=E_1\cup E_2\cup E(\mathbb{M})$, where $E(\mathbb{M})$ denotes the set of edges in $\mathbb{M}$. Note that for each $vv^{\prime}\in E(\mathbb{M})$, we have
\begin{equation}\label{edge-M}
\mathbb{P}\left(vv'\in E(V_0, V_1)\right)=1.
\end{equation}

Now we consider edges in $E_1$.  For each $uv\in E_1$, by Lemma \ref{th-main},  the probability that $u v$ lies in  $\left(V_0, V_1\right)$ is at least
$1 / 2+\epsilon \nabla_{u v}$, where
$$
\nabla_{uv}:=\frac{1}{\sqrt{d_G(u)+d_G\left(u^{\prime}\right)}}+\frac{1}{\sqrt{d_G(v)+d_G\left(v^{\prime}\right)}}.
$$
By the linearity of expectation, the expected number of the edges in $E_1$ contained in $\left(V_0, V_1\right)$ is at least
\begin{equation}\label{edge-e2}
\frac{\left|E_1\right| } 2+\epsilon \sum_{u v \in E_1} \nabla_{u v}.
\end{equation}

At last we consider edges in $E_2$. Suppose that $f,f'$ be two edges in some $E_{ij}$. Then $f$ and $f'$ should share a endpoint (See  Figure \ref{four-quasi-triangle}), and we call the subgraph induced by the vertices in $f$ and $f'$ a \emph{quasi-triangle}. Regardless of the values of  $h,h'$,
exactly one edge in $\{f,f'\}$ will fall into $\left(V_0, V_1\right)$. Thus,
\begin{equation}\label{edge-e1}
\mathbb{P}\left(|E(V_0,V_1)\cap E_2|=|E_2|/2\right)=1.
\end{equation}

\captionsetup{justification=centering}

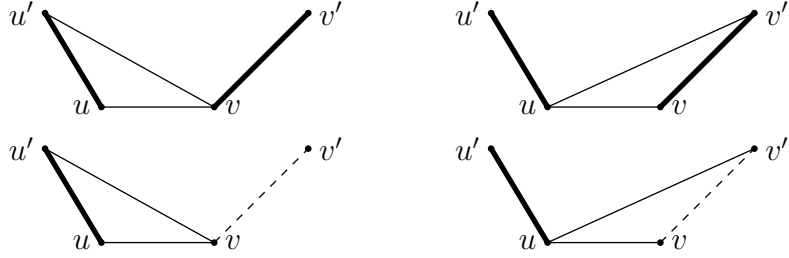
\begin{figure}[H]\label{four-quasi-triangle}
    \centering
    \begin{subfigure}[b]{0.24\textwidth}
        \centering
        \begin{tikzpicture}[scale=0.5,line width=0.5pt]
            \draw (-0.5,2.5) [fill=black] circle(0.07)node  [left]{$u'$};
            \draw (1,0) [fill=black] circle(0.07)node  [left] {$u$};
            \draw (4,0) [fill=black] circle(0.07)node  [right] {$v$};
            \draw (6.5,2.5) [fill=black] circle(0.07)node  [right] {$v'$};
            \draw [line width=2pt](-0.5,2.5) --(1,0);
            \draw[line width=0.5pt](1,0)--(4,0);
            \draw[line width=2pt](4,0)--(6.5,2.5);
            \draw[line width=0.5pt](-0.5,2.5)--(4,0);
        \end{tikzpicture}
    \end{subfigure}
    \hspace{2cm}
    \begin{subfigure}[b]{0.24\textwidth}
        \centering
        \begin{tikzpicture}[scale=0.5,line width=0.5pt]
            \draw (-0.5,2.5) [fill=black] circle(0.07)node  [left]{$u'$};
            \draw (1,0) [fill=black] circle(0.07)node  [left] {$u$};
            \draw (4,0) [fill=black] circle(0.07)node  [right] {$v$};
            \draw (6.5,2.5) [fill=black] circle(0.07)node  [right] {$v'$};
            \draw [line width=2pt](-0.5,2.5) --(1,0);
            \draw[line width=0.5pt](1,0)--(4,0);
            \draw[line width=2pt](4,0)--(6.5,2.5);
            \draw[line width=0.5pt](1,0)--(6.5,2.5);
        \end{tikzpicture}
    \end{subfigure}
    \hspace{2cm}
    \begin{subfigure}[b]{0.24\textwidth}
        \centering
        \begin{tikzpicture}[scale=0.5,line width=0.5pt]
            \draw (-0.5,2.5) [fill=black] circle(0.07)node  [left]{$u'$};
            \draw (1,0) [fill=black] circle(0.07)node  [left] {$u$};
            \draw (4,0) [fill=black] circle(0.07)node  [right] {$v$};
            \draw (6.5,2.5) [fill=black] circle(0.07)node  [right] {$v'$};
            \draw [line width=2pt](-0.5,2.5) --(1,0);
            \draw[line width=0.5pt](1,0)--(4,0);
            \draw[dashed](4,0)--(6.5,2.5);
            \draw[line width=0.5pt](-0.5,2.5)--(4,0);
        \end{tikzpicture}
    \end{subfigure}
    \hspace{2cm}
    \begin{subfigure}[b]{0.24\textwidth}
        \centering
        \begin{tikzpicture}[scale=0.5,line width=0.5pt]
            \draw (-0.5,2.5) [fill=black] circle(0.07)node  [left]{$u'$};
            \draw (1,0) [fill=black] circle(0.07)node  [left] {$u$};
            \draw (4,0) [fill=black] circle(0.07)node  [right] {$v$};
            \draw (6.5,2.5) [fill=black] circle(0.07)node  [right] {$v'$};
            \draw [line width=2pt](-0.5,2.5) --(1,0);
            \draw[line width=0.5pt](1,0)--(4,0);
            \draw[dashed](4,0)--(6.5,2.5);
            \draw[line width=0.5pt](1,0)--(6.5,2.5);
        \end{tikzpicture}
    \end{subfigure}
    \caption{Four kinds quasi-triangles}\label{four-quasi-triangle}
\end{figure}

Combining \eqref{edge-M}, \eqref{edge-e2} and \eqref{edge-e1},
\begin{align}\label{expaction-cut-main-thm}
\mathbb{E}\left(e\left(V_0, V_1\right)\right)
 \geq  |E(\mathbb{M})|+\frac{\left|E_1\right|}{2}+\frac{\left|E_2\right|}{2}+\epsilon \sum_{u v \in E_1} \nabla_{uv}
\geq \frac{m}{2}+\epsilon\sum_{u v \in E_1} \nabla_{u v}.
\end{align}

To complete the proof, we need bound $\sum_{u v \in E_1} \nabla_{u v}$.
By the  $C_4$-freeness of $G$, each pair in $\mathbb{M}$ is contained in at most one quasi-triangle, which  implies that
\begin{equation}\label{number-edge-E2}
|E_2|\le 2*n/2=n.
\end{equation}
For $e_i=uu'\in \mathbb{M}$, let $k_{i}$ denote the number of edges in $E_{2}$ which are incident with $u$ or $u'$. Then by \eqref{number-edge-E2},
\begin{equation}\label{sum-ki}
\sum_{i=1}^{n/2}k_i=2|E_2|\le 2n.
\end{equation}
Let $H$ be the spanning subgraph of $G$ with the edge set $E_1$. Then
\begin{equation*}
d_H(u)+d_H\left(u'\right)\ge d_G(u)+d_G\left(u^{\prime}\right)-k_i-2 \geq 0.
\end{equation*}
Using double counting, we have
$$
\sum_{u v \in E_1} \nabla_{u v}=\sum_{u u^{\prime} \in \mathbb{M}} \frac{d_H(u)+d_H\left(u^{\prime}\right)}{\sqrt{d_G(u)+d_G\left(u^{\prime}\right)}}.
$$
Thus,
\[\sum_{u v \in E_1} \nabla_{u v}\ge \sum_{u u^{\prime} \in \mathbb{M}} \frac{d_G(u)+d_G\left(u^{\prime}\right)-k_i-2}{\sqrt{d_G(u)+d_G\left(u^{\prime}\right)}}
\ge \frac{1}{\sqrt{2}}\sum_{i=1}^{n}\sqrt{d_i}-\sum_{i=1}^{n / 2} \sqrt{k_i+2}.
\]
where the  last inequality holds as $\sqrt{d_{G}(u)+d_{G}(u')} \geq \sqrt{d_{G}(u)/2}+\sqrt{d_{G}(u')/2}$. On the other hand, by \eqref{sum-ki} and the Cauchy-Schwarz inequality,
\[
\sum_{i=1}^{n / 2} \sqrt{k_i+2} \leq\left(\frac{n}{2} \sum_{i=1}^{n / 2}\left(k_i+2\right)\right)^{1 / 2}  <\sqrt{2} n.
\]
We conclude that
$$
\sum_{uv \in E_{1}}  \nabla_{u v} \geq\frac{1}{\sqrt{2}}\left(\sum_{i=1}^{n}\sqrt{d_i}-2n\right) > \frac{1}{2}\sum_{i=1}^{n}\sqrt{d_i},
$$
where the last inequality is from $\sum_{i=1}^{n}\sqrt{d_i}>8(n-1)$ and $n$ is large enough. This together with \eqref{expaction-cut-main-thm} yields that $G$ has a bisection of size at least $m/2+\xi \sum_{i=1}^n \sqrt{d_i}$.



\section{Proof of Theorem \ref{the1.2}}\label{SEC:pf-thm-2}
In this section, we give a proof of Theorem \ref{the1.2} using a standard degenerate argument  initially  proposed by  Alon \cite{Alo1996}. For  $k\ge 3$, let $G$ be a $\{C_4,C_6,C_{2k}\}$-free graph with $m$ edges and minimum degree at least 2. Define $D=bm^{\frac{1}{k+1}}$, in which $b=b(k)$ will be chosen later.
We claim $G$ is $(D-1)$-degenerate for some  $b$. Otherwise,  $G$ contains a subgraph $G'$ with minimum degree at least $D$. Note that the number of vertices of $G'$ is
\[N \leq \frac{2m}D= \frac{2m^{k/(k+1)}}{b}.\]
This implies that the number of edges of $G'$ is
\[e(G')\ge \frac{DN}{2}\ge \left(\frac{bN}{2}\right)^{1+1/k}.\]
The well-known  Bondy-Simonovits Theorem \cite{Bondy1974}, on the  maximum number of edges in $C_{2k}$-free graphs, yields that there is a constant $c =c(k)>0$ such that $e(G')<(cN)^{1+1/k}$. We arrive a contradiction by choosing $b>2c$.

Now we  label $v_{1}, v_{2},\cdots , v_{n}$ of the vertices of $G$ such that $d^{+}_{i}<D$ for every $i$, where $d^{+}_{i}$ denote the number of neighbours $v_{j}$ of $v_{i}$ with $j<i$.
Then
\[\sum_{i=1}^{n}\sqrt{d(v_{i})}\geq \sum_{i=1}^{n}\sqrt{d^{+}_{i}}>\frac{\sum_{i=1}^{n}d^{+}_{i}}{\sqrt{D}}=\frac{m}{\sqrt{D}}\geq \frac{m^{(2k+1)/(2k+2)}}{\sqrt{b}}.\]
Hence, by Theorem \ref{the1.1}, $G$ has a bisection of size at least
\[\frac{m}{2}+\xi \sum_{i=1}^{n}\sqrt{d(v_{i})} \geq \frac{m}{2}+ \frac{\xi}{\sqrt{b}}m^{(2k+1)/(2k+2)}. \]



\section{Proof of Lemma \ref{th-main}}\label{SEC:pf-main-lemma}

In this section, we prove Lemma \ref{th-main}. Let $G=(V, E)$ be a $\{C_{4},C_{6}\}$-free graph with  $n$ vertices and minimum degree at least two, where $n$ is even. First we define a quasi-perfect matching $\mathbb{M}$ of $G$. Let $M_1$ be a maximum matching of $G$. Clearly, $W = V(G)\setminus V(M_1)$ is an independent set. Now we construct a new graph $H$ with the vertex set $W$ and $x,y\in W$ are adjacent in $H$ iff there exists $z\in V(M_1)$ such that $xz,yz\in E(G)$. Note that if $z$ exists, then it is unique as $G$ is $C_4$-free. Let $M_2$ be the maximum matching of $H$. We may choice $M_1$ such that the size of $M_2$ is as large as possible. Then, we randomly pair the vertices in $W\setminus V(M_2)$ independently. Denote $M_3$ for the set of these pairs. Clearly,
\[
\mathbb{M}=M_1 \cup M_2 \cup M_3=\left\{e_{1}, \ldots, e_{n/2}\right\}
\]
is  a quasi-perfect matching of $G$. To simplify the presentation, we use $e_v=vv'$ to denote the pair in $\mathbb{M}$ containing  $v\in V(G)$.

Suppose that $h,h'$ are  random functions, $(V_0,V_1)$ is the output bisection of $G$ by Algorithm \ref{ALGO:Two-stage-Random-Algorithm} and $uv$ be the unique edge of $G$ connecting $e_u$ and $e_v$.  Call $P=u'uvv'$ be the \emph{special $\mathbb{M}$-path} associated with $uv$. We may assume that $uu'\in E(G)$.
Define
\[
p_{uv} = \mathbb{P}(e_{u}\text{\ and \ } e_{v} \text{\  are stable} | \: h(u)\neq h(v))
\]
and
\[
q_{uv} =\mathbb{P}(e_{u}\text{\ and \ }  e_{v} \text{\  are stable}| \: h(u) = h(v)).
\]
The following proposition, given by Lin and Zeng \cite{Lin2021},  shows the probability that $uv$ falls in $(V_0,V_1)$. We mention that in Lin and Zeng's proof, they also need $G$ has perfect matchings.  But their proof is still valid  for graphs without perfect matchings, see,  \emph{e.g.},  \cite{Rao2022}, Section 4.
\begin{proposition}[Lin and Zeng \cite{Lin2021}]\label{lem-h'}
\begin{equation*}
\mathbb{P}\left(uv\in E(V_0,V_1)\right)=\mathbb{P}\left(h'(u)\neq h'(v)\right)=\frac{1}{2}+\frac{1}{4}\left(p_{uv}-q_{uv}\right).
\end{equation*}
\end{proposition}

In order to prove Lemma \ref{th-main},  the main task is to bound $p_{uv}-q_{uv}$ carefully. For $i,j \in \{0,1\}$, let
\[P_{ij}=\mathbb{P}(e_{u}\text{\ and \ }  e_{v} \text{\  are stable}| \:h(u)=i, h(v)=j).\]
It follows that
\[p_{uv}-q_{uv}=P_{01}+P_{10}-P_{00}-P_{11}.\]
Thus, it suffices to calculate $P_{ij}$ for every $i,j \in \{0,1\}$. By the definition of stability, it is easy to see that  $P_{01}=P_{10}$ and $P_{00}=P_{11}$.

Note that the stability of $e_{u}$ (or $e_{v}$) is determined by the $h$-values of the vertices in $N(u)\cup N(u')$ (or $N(v) \cup N(v')$). We can narrow this domain further. For $x,y\in V(G)$, let $N(x,y)$ denote the set of  common neighbors of $x$ and $y$ in $G$, and  let $S(x)$ be the vertex set obtained from $N(x)$ by deleting $x'$ for $xx'\in \mathbb{M}$, vertices in $N(x,x')$, $y$ and $y'$ with $yy' \in \mathbb{M}$. Clearly,  deleting vertices contribute nothing to $\sigma (xx')$. Thus,

\begin{proposition}\label{pro3.1}
For each $e\in \mathbb{M}$,
\[\sigma(e)=\sum_{x \in e}\left(|N_{1-h(x)}(x)\cap S(x)|-|N_{h(x)}(x)\cap S(x)|\right).\]
\end{proposition}

The crux of calculating the probability that $uv$ fall in $(V_0,V_1)$ lies in the fact that the stability of $e_{u}$ and $e_{v}$ may not be independent. For example, the stability of $e_{u}$ and $e_{v}$ will influence each other if there exist vertices  $x\in N(u)$ and  $x' \in N(v')$ with $xx' \in \mathbb{M}$. Here, we list all the possible cases where the stability of $e_{u}$ and $e_{v}$ is not independent  (see Figure \ref{seven-type-case}).
\begin{enumerate}[itemindent=1em]
\item[($S_1$)] vertices $x\in N(u)$ and $x' \in N(v')$ such that $xx' \in \mathbb{M}$;
\item[($S_2$)]  vertices $y \in N(v)$ and  $y'  \in N(u')$ such that $yy' \in \mathbb{M}$;
\item[($S_3$)] the vertex $p \in N(u,v)$;
\item[($S_4$)] the vertex $q \in N(u',v')$;
\item[($S_5$)] the vertex $r \in N(u,v')$;
\item[($S_6$)] vertices $z \in N(v)$ and  $z'  \in N(u)$ such that $zz' \in \mathbb{M}$;
\item[($S_7$)]  $u$ and $v$.
\end{enumerate}

\begin{minipage}[b]{0.9\textwidth}\label{seven-type-case}
\begin{center}
\begin{tikzpicture}[scale=1,line width=0.8pt]

\draw (3.5,0) [fill=black] circle(0.07)node  [below] {$x\rule{0pt}{9pt}$};
\draw (5,0) [fill=black] circle(0.07)node  [below] {$x'$};
\draw [line width=2pt, color=green] (1,0.5)--(2.5,0.5);
\draw [line width=2pt, color=red](3.5,0)--(5,0);

\draw (1,0.5) [fill=black] circle(0.07)node  [below] {$y_1'$};
\draw (2.5,0.5) [fill=black] circle(0.07)node  [below] {$y_1\rule{0pt}{9pt}$};
\draw (1,-0.5) [fill=black] circle(0.07)node  [below] {$y_2'$};
\draw (2.5,-0.5) [fill=black] circle(0.07)node  [below] {$y_2\rule{0pt}{9pt}$};
\draw [line width=2pt, dashed, color=blue](1,-0.5) --(2.5,-0.5);

\draw (-0.5,2.5) [fill=black] circle(0.07)node  [left]{$u'$};
\draw (2,2) [fill=black] circle(0.07)node  [above] {$u$};
\draw (4,2) [fill=black] circle(0.07)node  [above] {$v$};
\draw (6.5,2.5) [fill=black] circle(0.07)node  [right] {$v'$};
\draw [line width=2pt](-0.5,2.5) --(2,2);
\draw[line width=0.5pt](2,2)--(4,2);
\draw[line width=2pt, dashed](4,2)--(6.5,2.5);

\draw [line width=0.5pt](-0.5,2.5)--(1,0.5);
\draw  [line width=0.5pt](2,2)--(3.5,0);
\draw [line width=0.5pt](4,2)--(2.5,0.5);
\draw [line width=0.5pt](6.5,2.5)--(5,0);

\draw [line width=0.5pt](-0.5,2.5)--(1,-0.5);
\draw [line width=0.5pt](4,2)--(2.5,-0.5);

\draw (3,3) [fill=black] circle(0.07)node  [above] {$p$};
\draw (3,5) [fill=black] circle(0.07)node  [above] {$q$};
\draw (3,4) [fill=black] circle(0.07)node  [above] {$r$};

\draw [line width=0.5pt](2,2)--(3,3)--(4,2);
\draw [line width=0.5pt](-0.5,2.5)--(3,5)--(6.5,2.5);
\draw [line width=0.5pt](2,2)--(3,4)--(6.5,2.5);

\draw (2,1.1) [fill=black] circle(0.07)node  [below] {$z'$};
\draw (4,1.1) [fill=black] circle(0.07)node  [below] {$z\rule{0pt}{9pt}$};
\draw [line width=2pt, dashed, color=yellow] (2,1.1) -- (4,1.1);
\draw[line width=0.5pt](2,1.1)--(2,2);
\draw[line width=0.5pt](4,1.1)--(4,2);

\draw[line width=0.3pt](0,-1.3) rectangle (6,5.7);
\end{tikzpicture}
\captionsetup{font=footnotesize}

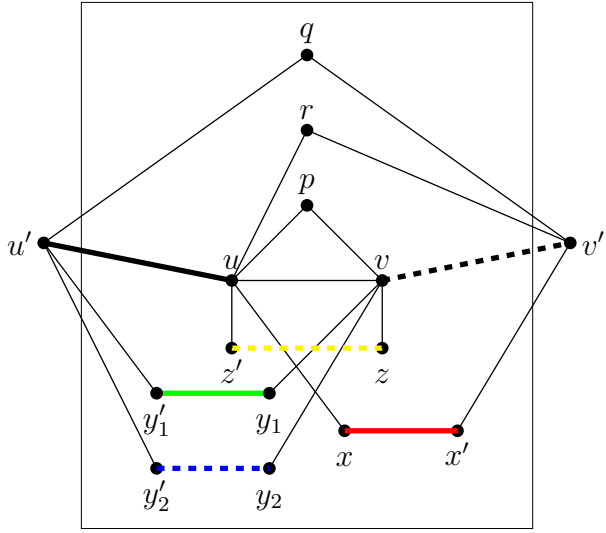
\captionof{figure}{Vertices in $S(P)$}\label{seven-type-case}
\end{center}
\end{minipage}

The labels of vertices from different types ($S_1$)-($S_7$) are pairwise independent. For $P=u'uvv'$, let $S(P)$ be the  set of vertices of types ($S_1$)-($S_7$). Clearly, vertices in $S(P)$ will affect the stability of $uu'$ and $vv'$. We illustrate this influence is small by considering the number of vertices in $S(P)$. 
For $i\in [6]$, let $k_{i}$ be the number  of vertices  (or pairs in ($S_1$), ($S_2$) and ($S_6$)) of type ($S_i$). Moreover, for $1\le i\le 2$, let $k_{ij}$ be the number of vertex pairs in $M_j$ of type ($S_i$).
Clearly, $k_1=k_{11}+k_{12}+k_{13}$ and $k_2=k_{21}+k_{22}+k_{23}.$

We are going to bound  $k_i$ or $k_{ij}$, which is vital in our subsequent analysis. Recall that $G$ contains neither $C_4$ nor $C_6$.
Then we have
\[k_{11},k_{21},k_3,k_4,k_5\in \{0,1\}.\]
It is readily verifiable that  $k_{22}=0$. Using the choice of $\mathbb{M}$, we emphasize that $k_{23}\le 1$. Otherwise, suppose that $y_1,y_2\in N(v)$, $y_1',y_2'\in N(u')$ and $y_1y_1',y_2y_2'\in M_3$. This means that $y_1y_2, y_1'y_2'\in E(H)$, contradicts to the fact that $M_2$ is a maximum matching of $H$. Recall that $uu'\in M_1$. If $vv'\in M_2\cup M_3$, then $k_{12}+k_{13}=0$. Otherwise, $k_5=(k_{12}+k_{13}) \times (k_{22}+k_{23})=0$ by the choice of $\mathbb{M}$ and $C_4$-freeness. In this case, we may assume that $k_{12}+k_{13}=0$ by symmetry.

We  summarize these points as follows. Note that  (III) and (IV) are not hard to verify by the choice of $\mathbb{M}$ and $\{C_4,C_6\}$-freeness of $G$, and (V) can be deduced by (I)-(III). 
\begin{proposition}\label{spc}
The following statements hold.
\begin{enumerate}[itemindent=1em]
\item[(I)] $k_{12}=k_{13}=k_{22}=0$;
\item[(II)] $k_1=k_{11},k_{21},k_{23},k_3,k_4,k_5\in \{0,1\}$;
\item[(III)] $k_{11} \times k_{23}=k_{11} \times k_4=k_{21} \times k_3=0$;
\item[(IV)]  If $vv'\in M_1$, then $k_{3} \times k_4=k_{21} \times k_4=0$.
\item[(V)]  $-1 \leq k_{1}+k_{2}+k_{3}-k_{4}\leq 2.$
\end{enumerate}
\end{proposition}

For each  $x \in V(P)$, define 
\[
T(x)=S(x)\cap S(P) \text{ and } T_{i}(x)=\{z \in T(x):h(z)=i\}
\]
for $i=0,1$. Note that $T(u) \cup T\left(u^{\prime}\right)$  and  $T(v) \cup T\left(v^{\prime}\right)$  are symmetry in the following sense: for each  $x \in T(u) \cup T\left(u^{\prime}\right)$, there is a  $x^{\prime} \in T(v) \cup T\left(v^{\prime}\right)$  such that  $x=x^{\prime}$  or  $xx^{\prime} \in \mathbb{M}$  or  $x=v$ and  $x^{\prime}=u$. For an integer $k\ge 0$ and $x\in S(P)$, define
\[
\{x\}_{k}=\begin{cases}
\emptyset & k=0 \\
x_1,\cdots,x_{k} & k\ge 1
\end{cases}.
\]
We may assume that  (see Figure \ref{seven-type-case})
\begin{equation*}
S(P)=\left\{u,v,\{x\}_{k_1},\{x'\}_{k_{1}},\{y\}_{k_{2}},\{y'\}_{k_{2}},\{p\}_{k_{3}},\{q\}_{k_{4}},\{r\}_{k_{5}},\{z\}_{k_6},\{z'\}_{k_6}\right\}.
\end{equation*}
Consequently,
\[T(u)=\{v,\{x\}_{k_{1}},\{p\}_{k_{3}},\{r\}_{k_{5}},\{z'\}_{k_6}\},\quad T(v)=\{u,\{y\}_{k_{2}},\{p\}_{k_{3}},\{z\}_{k_6}\}\]
and
\[T(u')=\{\{y'\}_{k_{2}},\{q\}_{k_{4}}\},\quad T(v')=\{\{x'\}_{k_{1}},\{q\}_{k_{4}},\{r\}_{k_{5}}\}.\]
Clearly, 
\begin{equation}\label{Lem:t(u)+t(u')}
|T(u)|+|T(u')|=|T(v)|+|T(v')|=\sum_{i=1}^6k_i+1.
\end{equation}
The basic barrier of Lin and Zeng's method  \cite{Lin2021} in here is $\sum_{i=1}^6k_i$ maybe large. The idea of the following proof treats those uncontrollable $k_i$  as variables. Suppose that there are exactly $a$ vertices in $\{\{x\}_{k_{1}}\}$ with $h$-value 0, $b$ vertices in $\{\{y\}_{k_{2}}\}$ with $h$-value 0, $c$ vertices in $\{\{p\}_{k_{3}}\}$ with $h$-value 0, $d$ vertices in $\{\{q\}_{k_{4}}\}$ with $h$-value 0, $e$ vertices in $\{\{r\}_{k_{5}}\}$ with $h$-value 0, and $f$ vertices in $\{\{z\}_{k_{6}}\}$ with $h$-value 0 (see Table \ref{TAB:h-value-SP}).

\begin{table}[H]
\centering
\caption{Labels of vertices in $S(P)$}\label{TAB:h-value-SP}
\begin{tabular}{|p{0.5cm}<{\centering}|p{1cm}<{\centering}p{1cm}<{\centering}|p{1cm}<{\centering}p{1cm}<{\centering}|p{1cm}<{\centering}|p{1cm}<{\centering}|p{1cm}<{\centering}|p{1cm}<{\centering}p{1cm}<{\centering}|}\hline
  &$\{x\}_{k_{1}}$ &$\{x'\}_{k_{1}}$ &$\{y\}_{k_{2}}$ &$\{y'\}_{k_{2}}$ &$\{p\}_{k_{3}}$  &$\{q\}_{k_{4}}$  &$\{r\}_{k_{5}}$ &$\{z\}_{k_6}$ &$\{z'\}_{k_6}$\\ \hline
0&$a$ &$k_{1}-a$ &$b$ &$k_{2}-b$ &$c$  &$d$  &$e$  &$f$  &$ k_6-f$              \\ \hline
1&$k_{1}-a$ &$a$ &$k_{2}-b$ &$b$  &$k_{3}-c$  &$k_{4}-d$ &$k_{5}-e$  &$ k_6-f$ &$f$  \\ \hline
\end{tabular}
\end{table}

For $xx'\in \{uu', vv'\}$, we calculate the probability that $xx'$ is stable under the condition $h(u)=i,h(v)=j$. Suppose that there are exactly $\alpha +|T_{h(x)}(x)|$ vertices in $S(x)$ whose $h$-value is equal to $h(x)$, and $\beta +|T_{h(y)}(x')|$ vertices in $S(x')$ whose $h$-value is equal to $h(x')$. By \eqref{Def:sigma(vv')-stable} and Proposition \ref{pro3.1}, $xx'$ is stable iff
 \[(\alpha +|T_{h(x)}(x)|)+(\beta +|T_{h(x')}(x')|) \leq \frac{|S(x)|+|S(x')|}{2},\]
which equals
 \begin{align*}
 &\alpha + \beta \notag \\
 \leq & \frac{|S(x)|-|T(x)|+|S(x')|-|T(x')|}{2}-\left(|T_{h(x)}(x)|+|T_{h(x')}(x')|-\frac{|T(x)|+|T(x')|}{2}\right) \notag \\
 =& \frac{s(x,x')-t(x,x')}{2},
 \end{align*}
where
 \[s(x,x'):=|S(x)|-|T(x)|+|S(x')|-|T(x')|,\]
 and
 \[t(x,x'):=| T_{h(x)}(x)|+|T_{h(x')}(x')|-|T_{1-h(x)}(x)|-|T_{1-h(x')}(x')|.\]
We remark that 
\begin{equation}\label{label-d(x)+d(x')}
  d(x)+d(x')\ge |S(x)|+|S(x')|=s(x,x')+\sum_{i=1}^6 k_i.
\end{equation}

Observe  that $s(x,x')$ is a constant and $t(x,x')$ is a random variable for fixed $P$. For each $y \in V(P)$ and $z \in T(y)\backslash \{u,v\}$,  $\mathbb{P}(z \in T_{h(y)}(y))= \mathbb{P}(z\in T_{1-h(y)}(y))=1/2$. Let $\mathcal{T}_{xx'}^{ij}$ be the (multi-)set of all possible values of $t(x,x')$ under the condition  $h(u)=i$ and $h(v)=j$.
Then we have 
\begin{align*}
  &\mathbb{P}(xx' \text{\  is stable}| \; h(u)=i,h(v)=j)\\
  = &\mathbb{P}\left( \alpha + \beta\le \frac{s(x,x')-t(x,x')}{2} \Big|\; h(u)=i,h(v)=j\right)\\
    =&\frac{1}{2^{|S(x)|+|S(x')|-1}}\sum_{t(x,x') \in \mathcal{T}_{xx'}^{ij}}\sum_{\gamma=0}^{\lfloor \frac{s(x,x')-t(x,x')}{2}\rfloor}\sum_{\alpha + \beta
  =\gamma}\binom{|S(x)|-|T(x)|}{\alpha}\binom{|S(x')|-|T(x')|}{\beta} \\
  =&\frac{1}{2^{|T(x)|+|T(x')|-1}}\sum_{t(x,x') \in \mathcal{T}_{xx'}^{ij}}\frac{1}{2^{s(x,x')}}\sum_{\gamma =0}^{\lfloor \frac{s(x,x')-t(x,x')}{2}\rfloor}\binom{s(x,x')}{\gamma}.
\end{align*}
Here we employ  the Vandermonde's convolution formula: for all positive integers $m_{1}, m_{2}$  and $n$, 
\[\sum_{k=0}^{n} \binom{m_{1}}{k} \binom{m_{2}}{n-k}=\binom{m_{1}+m_{2}}{n}.\]

Similarly,  we can determine $P_{ij}$.  Due to the symmetry  of $S(P)$,  once $h$-values of  vertices in either $T(u)\cup T(u')$ or $T(v)\cup T(v')$ are fixed, then  $h$-values of the other part are determined. For example, if there exists $yy' \in \mathbb{M}$ such that  $y \in T(u) \cup T(u')$ and $y' \in T(v)\cup T(v')$, then  $h(y)=0$ (or $h(y)=1$) when $h(y')=1$ (or $h(y')=0$).
This implies that one of  $t(u, u')$ and $t(v,v')$ is determined by the other.
For simplicity, define 
\[
t_{1}=t(u,u'), t_{2}=t(v,v'), s_{1}=s(u,u') \text{ and } s_{2}=s(v,v'). 
\]
Recall that $B(N,r)=\frac{1}{2^{N}}\sum_{i=0}^{\lfloor r \rfloor}\binom{N}{i}$  and $\Phi(t_1,t_2)=B\left(s_{1}, \frac{s_{1}-t_{1}}{2}\right)B\left(s_{2}, \frac{s_{2}-t_{2}}{2}\right)$, we have 
\begin{align}\label{main-Pij}
P_{ij}&= \mathbb{P}(e_{u}\text{\ and \ }  e_{v} \text{\  are stable}| \: h(u)=i, h(v)=j)\notag\\
&=\frac{1}{2^{|T(u)|+|T(u')|-1}}\sum_{t_{1} \in \mathcal{T}_{uu'}^{ij}} \frac{1}{2^{s_{1}}} \sum_{\gamma_1 =0}^{\lfloor \frac{s_{1}-t_{1}}{2}\rfloor} \binom{s_{1}}{\gamma_1}\frac{1}{2^{s_{2}}} \sum_{\gamma_2=0}^{\lfloor \frac{s_{2}-t_{2}}{2}\rfloor} \binom{s_{2}}{\gamma_2}\notag\\
&=\frac{1}{2^{|T(u)|+|T(u')|-1}}\sum_{t_{1}\in \mathcal{T}_{u{u'}}^{ij}}B\left(s_{1}, \frac{s_{1}-t_{1}}{2}\right)B\left(s_{2}, \frac{s_{2}-t_{2}}{2}\right)\notag \\
&=\frac{1}{2^{\sum_{i=1}^6k_i}}\sum_{t_{1}\in \mathcal{T}_{u{u'}}^{ij}}\Phi(t_1,t_2), 
\end{align}
where the last equality follows from \eqref{Lem:t(u)+t(u')}. 

To complete the proof, we have to calculate $P_{01}$ and $P_{00}$, respectively. For $i,j \in \{0,1\}$, let $\mathcal{L}_{ij}$ be  the event that $h(u)=i$ and $h(v)=j$. If $\mathcal{L}_{01}$ happens, then  $h(u')=1$ and $h(v')=0$, which together with Table \ref{TAB:h-value-SP} yields that 
\begin{align}
t_{1}&=|T_{0}(u)|-|T_{1}(u)|+|T_{1}(u')|-|T_{0}(u')| \notag\\
     &=-1+a-(k_{1}-a)+b-(k_{2}-b)+c-(k_{3}-c) \notag\\
      &\ \ \ \ \ \ +(k_{4}-d)-d+e-(k_{5}-e)+(k_6-f)-f\notag\\
     &=2a+2b+2c-2d+2e-2f-k_{1}-k_{2}-k_{3}+k_{4}-k_{5}+k_6-1\label{t1}, 
\end{align}
and
\begin{align}
t_{2}&=|T_{1}(v)|-|T_{0}(v)|+|T_{0}(v')|-|T_{1}(v')| \notag\\
     &=-1+(k_{1}-a)-a+(k_{2}-b)-b+(k_{3}-c)-c \notag\\
     &\ \ \ \ \ \ +d-(k_{4}-d)+e-(k_{5}-e)+(k_6-f)-f\notag\\
     &=-2a-2b-2c+2d+2e-2f+k_{1}+k_{2}+k_{3}-k_{4}-k_{5}+k_6-1\label{t2}.
\end{align}
Similarly, if $\mathcal{L}_{00}$ happens, then 
\begin{align}
t'_{1}&=|T_{0}(u)|-|T_{1}(u)|+|T_{1}(u')|-|T_{0}(u')| \notag\\
      &=2a+2b+2c-2d+2e-2f-k_{1}-k_{2}-k_{3}+k_{4}-k_{5}+k_6+1\label{t'1}
\end{align}
and
\begin{align}
t'_{2}&=|T_{0}(v)|-|T_{1}(v)|+|T_{1}(v')|-|T_{0}(v')|\notag \\
      &=2a+2b+2c-2d-2e+2f-k_{1}-k_{2}-k_{3}+k_{4}+k_{5}-k_6+1\label{t'2}.
\end{align}
Combining \eqref{t1}--\eqref{t'2}, we have 
\begin{proposition}\label{t-t}
\[t_{2}=t_{1}-4(a+b+c-d)+2(k_{1}+k_{2}+k_{3}-k_{4}),\]
\[t_{1}'=t_{1}+2 \text{ and   } t_{2}'=-t_{2}.\]
\end{proposition}
By Proposition \ref{spc} (V), we have 
\[-1 \leq k_{1}+k_{2}+k_{3}-k_{4}\leq 2.\]
We divide our proof  into four cases according to $k_{1}+k_{2}+k_{3}-k_{4}$.  For convenience, let's write 
\[
\alpha = a+b+c-d.
\]

\noindent {\bf Case 1.}  $k_{1}+k_{2}+k_{3}-k_{4}=-1$.  \\

By Proposition \ref{spc}, we obtain $k_{1}=k_{2}=k_{3}=0$ and $k_4=1$. Consequently, $a=b=c=0$ and $\sum_{i=1}^6k_i=k_5+k_6+1$.
Let
\[t=2e-2f-k_5+k_6.\]
By Proposition \ref{t-t}, we get the  following table.
\begin{table}[H]
\centering
\caption{The values of $t_{1}$ and $t_{2}$}
\begin{tabular}{|p{1.5cm}<{\centering}|p{0.7cm}<{\centering}|p{2.35cm}<{\centering}p{2.35cm}<{\centering}|p{2.35cm}<{\centering}p{2.7cm}<{\centering}|}\hline
                   &         &\multicolumn{2}{c|}{$\mathcal{L}_{01}$}  &\multicolumn{2}{c|}{$\mathcal{L}_{00}$ }  \\ \hline
     $d$    &$\alpha$        &$t_{1}$           &$t_{2}$               &$t'_{1}$            &$t'_{2}$ \\ \hline

         0          &0    &$ t$         &$t-2$              &$t+2$         &$-t+2$ \\ \hline
         1          &-1   &$t-2$        &$t$               &$ t$           &$-t$  \\ \hline\end{tabular}
\end{table}

Note that $\mathcal{T}_{uu'}^{ij}$ is determined by $h$-values of vertices in $S(P)$ by its definition, which means that  we need to enumerate all possible $h$-values of vertices in $S(P)$ to calculate $P_{ij}$. As there are no vertices of type $(S_1)$--$(S_3)$ and exactly one vertex $q$ of type $(S_4)$, we can divide \eqref{main-Pij} into two term according the $h$-value of $q$ and have the following 
\begin{align}
P_{01}=P_{10}=&\frac{1}{2^{k_5+k_6+1}}\sum_{e=0}^{k_{5}}\sum_{f=0}^{k_6}\binom{k_{5}}{e}\binom{k_6}{f}\bigg[\Phi(t,t-2)+\Phi(t-2,t)\bigg] \label{Case1:P01}\\
=&\frac{1}{2^{k_5+k_6+1}}\sum_{e=0}^{k_{5}}\sum_{f=0}^{k_6}\binom{k_{5}}{k_5-e}\binom{k_6}{k_6-f}\bigg[\Phi(-t,-t-2)+\Phi(-t-2,-t)\bigg]\notag\\
=&\frac{1}{2^{k_5+k_6+1}}\sum_{e=0}^{k_{5}}\sum_{f=0}^{k_6}\binom{k_{5}}{e}\binom{k_6}{f}\bigg[\Phi(-t,-t-2)+\Phi(-t-2,-t)\bigg]\label{Case1:P01-refine}.
\end{align}
Note that the third equality holds by replacing $e$ with $k_5-e$ and $f$ with $k_6-f$. 
Similarly,
\begin{align}
P_{00}=P_{11}=&\frac{1}{2^{k_5+k_6+1}}\sum_{e=0}^{k_{5}}\sum_{f=0}^{k_6}\binom{k_{5}}{e}\binom{k_6}{f}\bigg[\Phi(t+2,-t+2)+\Phi(t,-t)\bigg]\label{Case1:P00}\\
=&\frac{1}{2^{k_5+k_6+1}}\sum_{e=0}^{k_{5}}\sum_{f=0}^{k_6}\binom{k_{5}}{e}\binom{k_6}{f}\bigg[\Phi(-t+2,t+2)+\Phi(-t,t)\bigg]\label{Case1:P00-refine}.
\end{align}
Combining \eqref{Case1:P01}, \eqref{Case1:P01-refine}, \eqref{Case1:P00}, \eqref{Case1:P00-refine}, we conclude that 
\begin{align}\label{Case1:puv-quv}
&p_{uv}-q_{uv}=2P_{01}-2P_{00}\notag\\
=&\frac{1}{2^{k_5+k_6+1}}\sum_{e=0}^{k_{5}}\sum_{f=0}^{k_6}\binom{k_{5}}{e}\binom{k_6}{f}\bigg[\Phi(t,t-2)+\Phi(t-2,t) +\Phi(-t,-t-2)+\Phi(-t-2,-t)\bigg]\notag\\
     -&\frac{1}{2^{k_5+k_6+1}}\sum_{e=0}^{k_{5}}\sum_{f=0}^{k_6}\binom{k_{5}}{e}\binom{k_6}{f}\bigg[\Phi(t+2,-t+2)+\Phi(t,-t)+\Phi(-t+2,t+2)+\Phi(-t,t)\bigg]\notag\\
\geq&
     \frac{1}{2^{k_5+k_6+1}}\sum_{e=0}^{k_{5}}\sum_{f=0}^{k_6}\binom{k_{5}}{e}\binom{k_6}{f}\cdot\Lambda, 
\end{align}
where 
\[
\Lambda:=\Phi(t,t-2)+\Phi(-t,-t-2)-\Phi(t+2,-t+2)-\Phi(-t+2,t+2), 
\]
and the last inequality holds by Lemma \ref{lembb-bb2} (iii).

Now we deal with $\Lambda$. If $t\geq 0$, by Lemmas  \ref{lem3.5} and \ref{lembb-bb2} (i),
\begin{align*}
\Lambda \geq & \Phi(-t,-t)-\Phi(-t+2,-t+2)\\
= & B\left(s_1,\frac{s_1+t}2\right)B\left(s_2,\frac{s_2+t}2\right)-B\left(s_1,\frac{s_1+t}2-1\right)B\left(s_2,\frac{s_2+t}2-1\right)
\\
= &\Theta\left(\frac{1}{2^{s_{1}}}\binom{s_{1}}{\frac{s_{1}+t}{2}}+\frac{1}{2^{s_{2}}}\binom{s_{2}}{\frac{s_{2}+t}{2}}\right)
= \Theta\left(\frac{1}{2^{s_{1}}}\binom{s_{1}}{\frac{s_{1}-t}{2}}+\frac{1}{2^{s_{2}}}\binom{s_{2}}{\frac{s_{2}-t}{2}}\right).
\end{align*}
Otherwise,  by Lemmas \ref{lem3.5} and \ref{lembb-bb2} (ii), we still have 
\begin{align*}
\Lambda\geq & \Phi(t,t)-\Phi(t+2,t+2)\\
= & B\left(s_1,\frac{s_1-t}2\right)B\left(s_2,\frac{s_2-t}2\right)-B\left(s_1,\frac{s_1-t}2-1\right)B\left(s_2,\frac{s_2-t}2-1\right)
\\
= &\Theta\left(\frac{1}{2^{s_{1}}}\binom{s_{1}}{\frac{s_{1}-t}{2}}+\frac{1}{2^{s_{2}}}\binom{s_{2}}{\frac{s_{2}-t}{2}}\right).
\end{align*}
In both cases, we have 
\begin{align}\label{Case1:bound-lambda}
\Lambda \ge \Theta\left(\frac{1}{2^{s_{1}}}\binom{s_{1}}{\frac{s_{1}-t}{2}}+\frac{1}{2^{s_{2}}}\binom{s_{2}}{\frac{s_{2}-t}{2}}\right).
\end{align}

Substituting the first term of the right hand in  \eqref{Case1:bound-lambda} and $t=2e-2f-k_5+k_6$ into \eqref{Case1:puv-quv} yields that 
\begin{align}\label{Case1:first-term}
  &\frac{1}{2^{k_5+k_6+1}}\sum_{e=0}^{k_{5}}\sum_{f=0}^{k_6}\binom{k_{5}}{e}\binom{k_6}{f}\left(\frac{1}{2^{s_{1}}}\binom{s_{1}}{\frac{s_{1}-t}{2}}\right)\notag\\
=&\frac{1}{2^{s_{1}+k_5+k_6+1}} \sum_{e=0}^{k_{5}}\sum_{f=0}^{k_{6}}\binom{k_{5}}{e}\binom{k_{6}}{k_6-f}\binom{s_{1}}{\frac{s_{1}+k_{5}-k_6}{2}-e+f}\notag\\ 
=&\frac{1}{2^{s_{1}+k_5+k_6+1}} \binom{s_{1}+k_5+k_6}{\frac{s_{1}+k_5+k_6}{2}}\notag\\
=&\Theta\left(\frac{1}{\sqrt{s_{1}+k_5+k_6+1}}\right)\geq \Theta\left(\frac{1}{\sqrt{d(u)+d(u')}}\right),
\end{align}
where the last inequality follows from \eqref{label-d(x)+d(x')}. Similarly, we can get 
\begin{align}\label{Case1:second-term}
\frac{1}{2^{k_5+k_6+1}}\sum_{e=0}^{k_{5}}\sum_{f=0}^{k_6}\binom{k_{5}}{e}\binom{k_6}{f}\left(\frac{1}{2^{s_{2}}}\binom{s_{2}}{\frac{s_{2}-t}{2}}\right)\geq \Theta\left(\frac{1}{\sqrt{d(v)+d(v')}}\right).
\end{align}
Combining \eqref{Case1:puv-quv}, \eqref{Case1:first-term}, \eqref{Case1:second-term}, we have 
\begin{align}\label{Lem:main-equaltiy-puv-q}
p_{uv}-q_{uv} =  \Theta\left(\frac{1}{\sqrt{d(u)+d(u')}}+\frac{1}{\sqrt{d(v)+d(v')}}\right).
\end{align}



\noindent {\bf Case 2.}  $k_{1}+k_{2}+k_{3}-k_{4}=1$.\\

We claim  $k_4=0$. Otherwise, $k_4=1$ and  then $k_1=0$ by Proposition \ref{spc} (I) and (III). This implies
\[k_1+k_{2}+k_{3}=k_{21}+k_{23}+k_{3}=2.\]
Using  Proposition \ref{spc} (II) and (III) yields $k_{21}+k_{3}\le 1$ and then 
$k_{23}=1$. This  implies that $vv'\in M_1$. However, by Proposition \ref{spc} (IV),  we have $k_3=k_{21}=0$, which implies $k_{21}+k_{23}+k_{3}\le 1$, a contradiction.

Let
\[t=2e-2f-k_5+k_6.\]
By Proposition \ref{t-t}, we get  the following table. 
\begin{table}[H]
\centering
\caption{The values of $t_{1}$ and $t_{2}$}
\begin{tabular}{|p{1.5cm}<{\centering}|p{0.7cm}<{\centering}|p{2.35cm}<{\centering}p{2.35cm}<{\centering}|p{2.35cm}<{\centering}p{2.7cm}<{\centering}|}\hline
                   &         &\multicolumn{2}{c|}{$\mathcal{L}_{01}$}  &\multicolumn{2}{c|}{$\mathcal{L}_{00}$ }  \\ \hline
     $a$/$b$/$c$    &$\alpha$        &$t_{1}$           &$t_{2}$               &$t'_{1}$            &$t'_{2}$ \\ \hline

         1          &1    &$ t$         &$t-2$              &$t+2$         &$-t+2$ \\ \hline
         0          &0   &$t-2$        &$t$               &$ t$           &$-t$  \\ \hline\end{tabular}
\end{table}
We can calculate $P_{ij}$ using the same arguments as in Case 1 and \eqref{Lem:main-equaltiy-puv-q} still holds. \\



\noindent {\bf Case 3.}  $k_{1}+k_{2}+k_{3}-k_{4}=2$.\\

Using Proposition \ref{spc}, it is easy to see  $k_4=0$. Recall that $k_1,k_3\le 1$. There are 4 choice of $\{k_1,k_2,k_3\}$ and then 16 choices of $\{a,b,c\}$ (see Table 4). Let
\[t=2e-2f-k_5+k_6-1.\]
By Proposition \ref{t-t}, we get the following table.

\begin{table}[H]\label{table-=2}
\centering
\caption{The values of $t_{1}$ and $t_{2}$}\label{table-ab/ac}
\begin{tabular}{|p{2.35cm}<{\centering}|p{1.cm}<{\centering}|p{2.3cm}<{\centering}p{2.3cm}<{\centering}|p{2.3cm}<{\centering}p{2.3cm}<{\centering}|}\hline
                      &         &\multicolumn{2}{c|}{$\mathcal{L}_{01}$}  &\multicolumn{2}{c|}{$\mathcal{L}_{00}$ }  \\ \hline
 $ab$/$ac/bc/b_1b_2$   &$\alpha$        &$t_{1}$           &$t_{2}$               &$t'_{1}$            &$t'_{2}$ \\ \hline
       00        &0              &$t-2$        &$t+2$            &$t$         &$-t-2$  \\ \hline
       10        &1              &$t$        &$t$            &$t+2$         &$-t$ \\ \hline
       01       &1               &$t$        &$t$            &$t+2$         &$-t$ \\ \hline
       11       &2               &$t+2$        &$t-2$            &$t+4$         &$-t+2$ \\ \hline
\end{tabular}
\end{table}

Combining Table \ref{table-ab/ac} and the argument as in Case 1,  \eqref{main-Pij} equals 
\begin{align*}
&P_{01}=P_{10}\\
=&\frac{1}{2^{k_5+k_6+2}}\sum_{e=0}^{k_{5}}\sum_{f=0}^{k_{6}}\binom{k_{5}}{e}\binom{k_{6}}{f} \left[\Phi(t-2,t+2)+2\Phi(t,t)+\Phi(t+2,t-2)  \right]\\
=&\frac{1}{2^{k_5+k_6+2}}\sum_{e=0}^{k_{5}}\sum_{f=0}^{k_{6}}\binom{k_{5}}{e}\binom{k_{6}}{f} \bigg[\Phi(-t-4,-t)+2\Phi(-t-2,-t-2)+\Phi(-t,-t-4)  \bigg].
\end{align*}
Note that the third equality holds by replacing $e$ with $k_5-e$ and $f$ with $k_6-f$.
Similarly,
\begin{align*}
&P_{00}=P_{11}\\
=&\frac{1}{2^{k_5+k_6+2}}\sum_{e=0}^{k_{5}}\sum_{f=0}^{k_{6}}\binom{k_{5}}{e}\binom{k_{6}}{f} \bigg[\Phi(t,-t-2)+2\Phi(t+2,-t)+\Phi(t+4,-t+2) \bigg]\\
=&\frac{1}{2^{k_5+k_6+2}}\sum_{e=0}^{k_{5}}\sum_{f=0}^{k_{6}}\binom{k_{5}}{e}\binom{k_{6}}{f}\bigg[\Phi(-t-2,t)+2\Phi(-t,t+2)+\Phi(-t+2,t+4) \bigg].
\end{align*}

Therefore,
\begin{align}\label{Case4:puv-quv}
 &p_{uv}-q_{uv}=2P_{01}-2P_{00}\notag\\
=&\frac{1}{2^{k_5+k_6+2}}\sum_{e=0}^{k_{5}}\sum_{f=0}^{k_{6}}\binom{k_{5}}{e}\binom{k_{6}}{f} \bigg[\Phi(t-2,t+2)+2\Phi(t,t)+\Phi(t+2,t-2)\notag\\
&+ \Phi(-t-4,-t)+2\Phi(-t-2,-t-2)+\Phi(-t,-t-4)    \bigg]\notag\\
&- \frac{1}{2^{k_5+k_6+2}}\sum_{e=0}^{k_{5}}\sum_{f=0}^{k_{6}}\binom{k_{5}}{e}\binom{k_{6}}{f}\bigg[\Phi(t,-t-2)+2\Phi(t+2,-t)+\Phi(t+4,-t+2)\notag\\
&+\Phi(-t-2,t)+2\Phi(-t,t+2)+\Phi(-t+2,t+4) \bigg] \notag\\
\geq &\frac{1}{2^{k_5+k_6+1}}\sum_{e=0}^{k_{5}}\sum_{f=0}^{k_{6}}\binom{k_{5}}{e}\binom{k_{6}}{f}\cdot\Lambda,
\end{align}
where
\[
\Lambda:=\Phi(t,t)+\Phi(-t-2,-t-2)-\Phi(t+2,-t)-\Phi(-t,t+2)
\]
and the last inequality holds by Lemma \ref{lembb-bb3}.
If $t \geq 0$, by Lemmas \ref{lem3.5} and \ref{lembb-bb} (i),
\begin{align*}
\Lambda \geq & \Phi(-t,-t)-\Phi(-t+2,-t+2)\\
= & B\left(s_1,\frac{s_1+t}2\right)B\left(s_2,\frac{s_2+t}2\right)-B\left(s_1,\frac{s_1+t}2-1\right)B\left(s_2,\frac{s_2+t}2-1\right)
\\
= &\Theta\left(\frac{1}{2^{s_{1}}}\binom{s_{1}}{\frac{s_{1}+t}{2}}+\frac{1}{2^{s_{2}}}\binom{s_{2}}{\frac{s_{2}+t}{2}}\right)
=\Theta\left(\frac{1}{2^{s_{1}}}\binom{s_{1}}{\frac{s_{1}-t}{2}}+\frac{1}{2^{s_{2}}}\binom{s_{2}}{\frac{s_{2}-t}{2}}\right).
\end{align*}
Otherwise, by  Lemmas \ref{lem3.5} and \ref{lembb-bb} (ii),
\begin{align*}
\Lambda \geq &  \Phi(t,t)-\Phi(t+2,t+2)\\
= & B\left(s_1,\frac{s_1-t}2\right)B\left(s_2,\frac{s_2-t}2\right)-B\left(s_1,\frac{s_1-t}2-1\right)B\left(s_2,\frac{s_2-t}2-1\right)
\\
= &\Theta\left(\frac{1}{2^{s_{1}}}\binom{s_{1}}{\frac{s_{1}-t}{2}}+\frac{1}{2^{s_{2}}}\binom{s_{2}}{\frac{s_{2}-t}{2}}\right).
\end{align*}
In both cases, we have
\begin{align}\label{Case4:bound-lambda}
\Lambda \ge \Theta\left(\frac{1}{2^{s_{1}}}\binom{s_{1}}{\frac{s_{1}-t}{2}}+\frac{1}{2^{s_{2}}}\binom{s_{2}}{\frac{s_{2}-t}{2}}\right).
\end{align}
Note that 
\begin{align}\label{Case4:first-term}
  &\frac{1}{2^{k_5+k_6+1}}\sum_{e=0}^{k_{5}}\sum_{f=0}^{k_6}\binom{k_{5}}{e}\binom{k_6}{f}\left(\frac{1}{2^{s_{1}}}\binom{s_{1}}{\frac{s_{1}-t}{2}}\right)\notag\\
=&\frac{1}{2^{s_{1}+k_5+k_6+1}} \sum_{e=0}^{k_{5}}\sum_{f=0}^{k_{6}}\binom{k_{5}}{e}\binom{k_{6}}{k_6-f}\binom{s_{1}}{\frac{s_{1}+k_{5}-k_6+1}{2}-e+f}\notag\\
=&\frac{1}{2^{s_{1}+k_5+k_6+1}} \binom{s_{1}+k_5+k_6}{\frac{s_{1}+k_5+k_6+1}{2}}\notag\\
=&\Theta\left(\frac{1}{\sqrt{s_{1}+k_5+k_6+1}}\right)\geq \Theta\left(\frac{1}{\sqrt{d(u)+d(u')}}\right),
\end{align}
where the last inequality follows from \eqref{label-d(x)+d(x')}. Similarly, we can get
\begin{align}\label{Case4:second-term}
\frac{1}{2^{k_5+k_6+1}}\sum_{e=0}^{k_{5}}\sum_{f=0}^{k_6}\binom{k_{5}}{e}\binom{k_6}{f}\left(\frac{1}{2^{s_{2}}}\binom{s_{2}}{\frac{s_{2}-t}{2}}\right)\geq \Theta\left(\frac{1}{\sqrt{d(v)+d(v')}}\right).
\end{align}
Combining \eqref{Case4:puv-quv}, \eqref{Case4:bound-lambda}, \eqref{Case4:first-term}, \eqref{Case4:second-term}, we have \eqref{Lem:main-equaltiy-puv-q} holds. \\



\noindent {\bf Case 4.}  $k_{1}+k_{2}+k_{3}-k_{4}=0$.\\

Let
\[t=2e-2f-k_5+k_6-1.\]
If $k_{4}=1$, then $k_{1}=0$ by Proposition \ref{spc}, and then $a=0$ and $k_{2}+k_{3}=1$.
By Proposition \ref{t-t}, we get the  following table and complete the proof by the same argument as in Case 3. 

\begin{table}[H]
\centering
\caption{The values of $t_{1}$ and $t_{2}$}
\begin{tabular}{|p{2.35cm}<{\centering}|p{1.cm}<{\centering}|p{2.3cm}<{\centering}p{2.3cm}<{\centering}|p{2.3cm}<{\centering}p{2.3cm}<{\centering}|}\hline
                      &         &\multicolumn{2}{c|}{$\mathcal{L}_{01}$}  &\multicolumn{2}{c|}{$\mathcal{L}_{00}$ }  \\ \hline
 $bd$/$cd$   &$\alpha$        &$t_{1}$           &$t_{2}$               &$t'_{1}$            &$t'_{2}$ \\ \hline
   00        &0              &$t$        &$t$            &$t+2$         &$-t$  \\ \hline
   10        &1              &$t+2$        &$t-2$            &$t+4$         &$-t+2$ \\ \hline
   01       &-1               &$t-2$        &$t+2$            &$t$         &$-t-2$ \\ \hline
   11       &0               &$t$        &$t$            &$t+2$         &$-t$ \\ \hline
\end{tabular}
\end{table}

Now suppose $k_4=0$. This means $k_{1}=k_{2}=k_{3}=0$ and then $a=b=c=d=0$.
By Proposition \ref{t-t},  we obtain the following table.
\begin{table}[H]
\centering
\caption{The values of $t_{1}$ and $t_{2}$}
\begin{tabular}{|p{1.2cm}<{\centering}|p{2.3cm}<{\centering}p{2.3cm}<{\centering}|p{2.3cm}<{\centering}p{3cm}<{\centering}|}\hline
                          & \multicolumn{2}{c|}{$\mathcal{L}_{01}$}  &\multicolumn{2}{c|}{$\mathcal{L}_{00}$ }  \\ \hline
  $\alpha$       &$t_{1}$           &$t_{2}$                       &$t'_{1}$            &$t'_{2}$ \\ \hline
   0  &$t$              &$t$                  &$t+2$                &$-t$  \\ \hline
\end{tabular}
\end{table}
By  \eqref{main-Pij},
\begin{align*}
P_{10}=P_{01}=&\frac{1}{2^{k_5+k_6}}\sum_{e=0}^{k_{5}}\sum_{f=0}^{k_{6}}\binom{k_{5}}{e}\binom{k_{6}}{f}\Phi(t,t)\\
=&\frac{1}{2^{k_5+k_6}}\sum_{e=0}^{k_{5}}\sum_{f=0}^{k_{6}}\binom{k_{5}}{e}\binom{k_{6}}{f}\Phi(-t-2,-t-2).
\end{align*}
Similarly,
\begin{align*}
P_{00}=P_{11}=&\frac{1}{2^{k_5+k_6}}\sum_{e=0}^{k_{5}}\sum_{f=0}^{k_{6}}\binom{k_{5}}{e}\binom{k_{6}}{f}\Phi(t+2,-t)\\
=&\frac{1}{2^{k_5+k_6}}\sum_{e=0}^{k_{5}}\sum_{f=0}^{k_{6}}\binom{k_{5}}{e}\binom{k_{6}}{f}\Phi(-t,t+2).
\end{align*}

From the above four equalities, we have 
\begin{align}\label{Case2:puv-quv}
p_{uv}-q_{uv}=\frac{1}{2^{k_5+k_6}}\sum_{e=0}^{k_{5}}\sum_{f=0}^{k_{6}}\binom{k_{5}}{e}\binom{k_{6}}{f}\cdot\Lambda,
\end{align}
where
\[
\Lambda:=\Phi(t,t)+\Phi(-t-2,-t-2)-\Phi(t+2,-t)-\Phi(-t,t+2).
\]
By Lemmas \ref{lem3.5}  and \ref{lembb-bb}, together with the similar argument as in Case 1, we have 
\begin{align}\label{Case2:bound-lambda}
\Lambda \ge \Theta\left(\frac{1}{2^{s_{1}}}\binom{s_{1}}{\frac{s_{1}-t}{2}}+\frac{1}{2^{s_{2}}}\binom{s_{2}}{\frac{s_{2}-t}{2}}\right).
\end{align}
Substituting  \eqref{Case2:bound-lambda} and $t=2e-2f-k_5+k_6-1$ into \eqref{Case2:puv-quv} and using  the similar argument as in Case 1, we can get
\begin{align}\label{Case2:first-term}
\frac{1}{2^{k_5+k_6}}\sum_{e=0}^{k_{5}}\sum_{f=0}^{k_{6}}\binom{k_{5}}{e}\binom{k_{6}}{f}\left(\frac{1}{2^{s_{1}}}\binom{s_{1}}{\frac{s_{1}-t}{2}}\right)\ge \Theta \left(\frac{1}{\sqrt{d(u)+d(u')}}\right), 
\end{align}
and 
 \begin{align}\label{Case2:second-term}
\frac{1}{2^{k_5+k_6}}\sum_{e=0}^{k_{5}}\sum_{f=0}^{k_6}\binom{k_{5}}{e}\binom{k_6}{f}\left(\frac{1}{2^{s_{2}}}\binom{s_{2}}{\frac{s_{2}-t}{2}}\right)\geq \Theta\left(\frac{1}{\sqrt{d(v)+d(v')}}\right).
\end{align}

Combining \eqref{Case2:puv-quv}, \eqref{Case2:bound-lambda}, \eqref{Case2:first-term}, \eqref{Case2:second-term}, we conclude that \eqref{Lem:main-equaltiy-puv-q} holds. This complete the proof of Lemma \ref{th-main}.

\section{Concluding remarks}

In \cite{Lin2021}, Lin and Zeng studied graphs whose Max-Bisection achieves the Shearer's bound systematically. The prefect matching condition is vital in their proof. In this paper, we replace it with minimum degree condition and solve an open problem. We believe that our method has potential applications in studying bisections of $H$-free graphs without perfect matching. 

In general the  $C_4$-free condition seems like the most natural to force a bisection larger than that guaranteed in a general graph. Combining Theorem \ref{cn} and the classical  Bondy-imonovits Theorem \cite{Bondy1974} on the number of edges in a $C_{2k}$-free graph, we know every $C_4$-free graph  with $m$ edges and minimum degree at least 2 has a bisection of size at least $m/2+\Omega(m^{2/3})$. On the other hand, Alon, Bollob\'{a}s,  Krivelevich and Sudakov \cite{Alo2003} showed that such a  graph has a bipartition of size at least $m/2+\Theta(m^{5/6})$. It is natural to find the maximal  $c$ such that every $C_4$-free graph  with $m$ edges and minimum degree at least 2 has a bisection of size at least $m/2+\Omega(m^{c})$.

It also would be interesting to generalise Theorem \ref{the1.1} to complete bipartite forbidden subgraphs. For $t\ge s\ge 2$, an accessible result, given by Hou and Wu \cite{Hou2021}, is the existence of a bisection of size at least $m/2+\Omega(n)$ in $K_{s,t}$-free graphs with $n$ vertices, $m$ edges and minimum degree $s$.

\section*{Appendix}
\begin{proof}[Proof of Lemma \ref{lembb-bb3}]

We prove Lemma \ref{lembb-bb3} using a coarse way. For  convenience, let 
\begin{align*}
\Lambda(t):= &\Phi(t-2,t+2)+\Phi(t+2,t-2)+\Phi(-t-4,-t)+\Phi(-t,-t-4)\\
-&\Phi(t,-t-2)-\Phi(-t+2,t+4)-\Phi(-t-2,t)-\Phi(t+4,-t+2). 
\end{align*}
It is easy to see that 
\[
\Lambda(t)=\Lambda(-t-2). 
\]
Thus, it suffices to consider $t\ge -1$. 
For any $x\in \mathbb{Z}$  and $i\in\{1,2\}$, let 
\[
f_i(x)=B\left(s_{i}, \frac{s_{i}-x}{2}\right).
\]
Note that $h_i(x)$ is an increasing function. 
If $t=-1$, then 
\begin{align*}
\Lambda(-1)=2\left(f_1(-3)f_2(1)+f_1(1)f_2(-3)-f_1(-1)f_2(-1)-f_1(3)f_2(3)\right). 
\end{align*}
In this case, we have $s_1,s_2\ge 3$.
For convince, let 
\[
\alpha_i=\frac{1}{2^{s_1}} \binom{s_1}{{\lfloor\frac{s_1-3}{2}\rfloor+i}} \text{ and } \beta_i=\frac{1}{2^{s_2}} \binom{s_2}{{\lfloor\frac{s_2-3}{2}\rfloor+i}} 
\]

\begin{align*}
\frac{1}{2}\Lambda(-1)=&\left(f_1(3)+\sum_{i=1}^3\alpha_i\right)\left(f_2(3)+\beta_1\right)+\left(f_1(3)+\alpha_1\right)\left(f_2(3)+\sum_{i=1}^3\beta_i\right)\\
-&\left(f_1(3)+\sum_{i=1}^2\alpha_i\right)\left(f_2(3)+\sum_{i=1}^2\beta_i\right)-f_1(3)f_2(3)\\
=&(\alpha_1+\alpha_3)f_2(3)+(\beta_1+\beta_3)f_1(3)+\Sigma\ge \Sigma, 
\end{align*}
where 
\begin{align*}
\Sigma=&(\alpha_1+\alpha_2+\alpha_3)\beta_1+\alpha_1(\beta_1+\beta_2+\beta_3)-(\alpha_1+\alpha_2)(\beta_1+\beta_2)\\
=& \alpha_3\beta_1+\alpha_1\beta_1+\alpha_1\beta_3-\alpha_2\beta_2. 
\end{align*}

In the following, we show $\Sigma\ge 0$  and then $\Lambda(-1)\ge 0$.  If $s_1$ and $s_2$ are odd, then $\alpha_1=\alpha_2$ and $\beta_1=\beta_2$. Thus,
\[\Sigma=\alpha_3\beta_1+\alpha_1\beta_3\ge 0.\]
If  $s_1$ is even and $s_2$ is odd, then $\beta_1=\beta_2$. Therefore,
\begin{align*}
\Sigma\ge(\alpha_3+\alpha_1-\alpha_2)\beta_1
=\frac{1}{2^{s_1}}\left( \binom{s_1}{{\frac{s_1-2}{2}}}+ \binom{s_1}{{\frac{s_1+2}{2}}}-\binom{s_1}{{\frac{s_1}{2}}}\right)\beta_1
\ge 0.
\end{align*}
The last inequality is from the fact that
\begin{align*}
& 2\times \binom{s_1}{{\frac{s_1-2}{2}}}\div\binom{s_1}{{\frac{s_1}{2}}}
=2 \times \frac{s_1!}{\frac{s_1-2}{2}!\frac{s_1+2}{2}!} \times \frac{\frac{s_1}{2}!\frac{s_1}{2}!}{s_1!}
=\frac{2s_1}{s_1+2}
\ge 1.
\end{align*}
Simarily, if  $s_1$ is odd and $s_2$ is even, then $\alpha_1=\alpha_2$ and so 
\begin{align*}
\Sigma\ge\alpha_1(\beta_3+\beta_1-\beta_2)
=\frac{1}{2^{s_2}}\left( \binom{s_2}{{\frac{s_2-2}{2}}}+ \binom{s_2}{{\frac{s_2+2}{2}}}-\binom{s_2}{{\frac{s_2}{2}}}\right)\alpha_1
\ge0.
\end{align*}
Suppose that  $s_1$ and $s_2$ are even. We have 
\begin{align*}
\Sigma=3 \times \left(\frac{1}{2^{s_1}} \binom{s_1}{{\frac{s_1-2}{2}}} \times \frac{1}{2^{s_2}} \binom{s_2}{{\frac{s_2-2}{2}}}\right)-\left(\frac{1}{2^{s_1}} \binom{s_1}{{\frac{s_1}{2}}} \times \frac{1}{2^{s_2}} \binom{s_2}{{\frac{s_2}{2}}}\right)
\geq  0, 
\end{align*}
where the last inequality holds as 
\begin{align*}
& 3 \times \left(\binom{s_1}{{\frac{s_1-2}{2}}} \times  \binom{s_2}{{\frac{s_2-2}{2}}}\right) \div \left(\binom{s_1}{{\frac{s_1}{2}}} \times \binom{s_2}{{\frac{s_2}{2}}}\right)\\
=&3 \times \frac{s_1!}{\frac{s_1-2}{2}!\frac{s_1+2}{2}!} \times \frac{s_2!}{\frac{s_2-2}{2}!\frac{s_2+2}{2}!} \times \frac{\frac{s_1}{2}!\frac{s_1}{2}!}{s_1!} \times \frac{\frac{s_2}{2}!\frac{s_2}{2}!}{s_2!}\\
=&\frac{3s_1s_2}{(s_2+2)\times(s_2+2)}
\ge  1.
\end{align*}
Thus, 
$\Sigma\ge 0$ and then $\Lambda(-1)\ge 0$. \\

Suppose that $t\ge 0$. As above, define
\[
\alpha_i=\sum_{j=1}^{i}\frac{1}{2^{s_1}} \binom{s_1}{{\lfloor\frac{s_1-t-4}{2}\rfloor+j}} \text{ and } \beta_i=\sum_{j=1}^{i}\frac{1}{2^{s_2}} \binom{s_2}{{\lfloor\frac{s_2-t-4}{2}\rfloor+j}}. 
\]

Note that $\Phi(t_1,t_2)=f_1(t_1)f_2(t_2)$. 
\begin{align}\label{lembda-t-1-2}
\Lambda(t)=&f_1(t-2)f_2(t+2)+f_1(t+2)f_2(t-2)+f_1(-t-4)f_2(-t)+f_1(-t)f_2(-t-4)\notag \\
-&f_1(t)f_2(-t-2)-f_1(-t+2)f_2(t+4)-f_1(-t-2)f_2(t)-f_1(t+4)f_2(-t+2)
\end{align}
Each term in the right hand of \eqref{lembda-t-1-2} can be expressed into the form 
\[
\left(f_1(t+4)+\alpha_i\right)\left(f_2(t+4)+\beta_j
\right)
\]
for some nonnegative integers $i,j$. For example, 
\[
f_1(t-2)f_2(t+2)=\left(f_1(t+4)+\alpha_3\right)\left(f_2(t+4)+\beta_1\right)
\]
and 
\[
f_1(t+2)f_2(t-2)=\left(f_1(t+4)+\alpha_{1}\right)  \left(f_2(t+4)+\beta_3\right). 
\]
The others terms is similar and we  omit it. Substituting those into \eqref{lembda-t-1-2}, we conclude that 
\[
\Lambda(t)=f_1(t+4)\times \Sigma_1+f_2(t+4)\times \Sigma_2+ \Sigma_3, 
\]
where
\[
\Sigma_1:=\beta_1+\beta_3+\beta_{t+2}+\beta_{t+4}-\beta_{t+3}-\beta_{2}-\beta_{t+1}, 
\]
\[
\Sigma_2:=\alpha_3+\alpha_{1}+\alpha_{t+4}+\alpha_{t+2}-\alpha_{2}-\alpha_{t+1}-\alpha_{t+3}, 
\]
and 
\begin{align*}
\Sigma_3:&=\alpha_3\beta_1+\alpha_{1}\beta_3+\alpha_{t+4}\beta_{t+2}+\alpha_{t+2}\beta_{t+4} -\alpha_{2}\beta_{t+3}-\alpha_{t+3}\beta_{2}\\
&\ge \alpha_{t+4}\beta_{t+2}-\alpha_{t+3}\beta_{2}+\alpha_{t+2}\beta_{t+4}-\alpha_{2}\beta_{t+3}
\\ & \ge 0.
\end{align*}
Recall that  $\alpha_i$ and $ \beta_i$ are increasing with respect to $i$, we also have $\Sigma_1, \Sigma_2\ge 0$. This completes the proof of Lemma \ref{lembb-bb3}.
\end{proof}

\end{document}